\documentclass[a4paper,12pt]{article}     
\usepackage{amsmath,amscd,amssymb}        
\usepackage{latexsym}                     
\usepackage[english, german]{babel}                
\usepackage{xypic}

\usepackage{theorem}                 
\usepackage{color}

\newtheorem{theorem}{Theorem}
\newtheorem{corollary}{Corollary}
\newtheorem{proposition}{Proposition}
\newtheorem{lemma}{Lemma}

\theorembodyfont{\rmfamily}
\newtheorem{remark}{Remark}
\newtheorem{example}{Example}

\newenvironment{definition}
{\smallskip\noindent{\bf Definition\/}:}{\smallskip\par}

\newenvironment{remarks}
{\smallskip\noindent{\bf Remarks\/}.}{\smallskip\par}

\newenvironment{proof}{\begin{ProofwCaption}{Proof}}{\end{ProofwCaption}}
\newenvironment{proof*}[1]{\begin{ProofwCaption}{{#1}}}{\end{ProofwCaption}}
\newenvironment{ProofwCaption}[1]%
  {\addvspace\theorempreskipamount \noindent{\it #1.}\rm}%
  {\qed \par \addvspace\theorempostskipamount}
\newcommand{\qedsymbol}{{\rm $\Box$}}
\newcommand{\qed}{\hfill\qedsymbol}

\newcommand{\e}{\widehat{e}}

\newcommand{\CC}{{\mathbb C}}

\newcommand{\QQ}{{\mathbb Q}}

\newcommand{\RR}{{\mathbb R}}
\newcommand{\ZZ}{{\mathbb Z}}

\newcommand{\calH}{{\mathcal H}}

\title{Orbifold Milnor lattice and orbifold intersection form}
\author{Wolfgang Ebeling and Sabir M.~Gusein-Zade
\thanks{Partially supported by DFG (Eb 102/8-1). The work of the second author
(Sections~\ref{sect:qhg}, \ref{sect:rings}, \ref{sect:Seifert} and~\ref{sect:inter}) was supported
by the grant 16-11-10018 of the Russian Science Foundation.
AMS 2010 Math. Subject Classification: 14R20, 57R18, 32S05, 58K65, 58K70.
}
}
\date{}

\begin{document}
\selectlanguage{english}

\maketitle

\begin{abstract}
For a germ of a quasihomogeneous function with an isolated critical point at the origin invariant with respect to an appropriate action of a finite abelian group, H.~Fan, T.~Jarvis, and Y.~Ruan defined the so-called quantum cohomology group. It is considered as the main object of the quantum singularity theory (FJRW-theory). 
We define orbifold versions of the monodromy operator on the quantum (co)homology group, of the Milnor lattice, of the Seifert form and of the intersection form. We also describe some symmetry properties of invariants of invertible polynomials refining the known ones.
\end{abstract}

\section{Introduction} \label{sect:Intro}
In \cite{Ruan_etal}, for a germ $f$ of a quasihomogeneous function with an isolated critical point at
the origin invariant with respect to an appropriate action of a finite abelian group $G$ (an admissible one),
H.~Fan, T.~Jarvis, and Y.~Ruan defined the so-called quantum cohomology group $\calH_{f,G}$. This group
is related to the vanishing cohomology groups of Milnor fibres of restrictions of $f$ to fixed point sets
of elements of $G$. The quantum cohomology group is considered as the main object of the quantum singularity
theory (FJRW-theory). In \cite{Ruan_etal}, the authors study some structures on it which generalize similar
structures in the usual singularity theory.

An important role in singularity theory is played by such concepts as the (integral) Milnor lattice,
the monodromy operator, the Seifert form and the intersection form. Analogues of these concepts have not
yet been considered in the FJRW-theory.

Here we define an orbifold version of the monodromy operator on the quantum (co)homology group
$\calH_{f,G}$ and a lattice $\Lambda_{f,G}$ in $\calH_{f,G}$ which is invariant with respect to
the orbifold monodromy operator and is considered as an orbifold version of the Milnor lattice.
The action of the orbifold monodromy operator on it can be considered as an analogue of the integral
monodromy operator. Moreover, we define orbifold versions of the Seifert form and of the intersection form.
We show that they are related by equations similar to those in the classical case.

To define these concepts we introduce the language of group rings. An appropriate change of the basis in
the group ring allows to give a decomposition of a certain extension of the quantum (co)homology group
$\calH_{f,G}$ into parts isomorphic to (co)homology groups of certain suspensions of the restrictions of
the function $f$ to fixed point sets. This permits to define analogues of the Seifert and intersection form
on this extension. We show that the intersection of this decomposition with the quantum (co)homology group
respects the relations between the monodromy, the Seifert and the intersection form. This defines these
concepts on the quantum (co)homology group.

In the last section, we shall consider some examples. These examples are chosen in the class of  so-called invertible polynomials (orbifold
Landau--Ginzburg models in the terminology of \cite{BH1,BH2}). We also consider the Berglund--H\"ubsch--Henningson duality of pairs $(f,G)$ where $f$ is an invertible polynomial and discuss the behaviour of the Milnor lattice under this mirror symmetry.

The authors are grateful to A.~Takahashi for very useful discussions permitting them to understand some
peculiarities in the definitions which they initially missed. We would like to thank the referee for carefully reading our paper and for valuable comments which helped to improve the paper.

\section{Quantum cohomology group} \label{sect:qhg}
Let $f:(\CC^n,0)\to(\CC,0)$ be a germ of a holomorphic function with an isolated critical point at
the origin and let $G$ be a finite abelian group acting faithfully on $(\CC^n,0)$ and preserving $f$.
(Without loss of generality one can assume that the action of $G$ is linear and diagonal.)

An important example is an invertible polynomial $f$ in $n$ variables with a subgroup $G$
of its maximal diagonal symmetry group $G_f$: \cite{BH2}. To a pair $(f,G)$ of this sort, 
P.~Berglund, T.~H\"ubsch, and M.~Henningson \cite{BH1,BH2} defined a dual pair $(\widetilde{f}, \widetilde{G})$. (Another description of the dual group was given by Krawitz \cite{Krawitz}.) This construction plays an
important role in mirror symmetry.

For a subgroup $K\subset G$, let $(\CC^n)^K$ be the fixed point set 
$\{x\in \CC^n \, \vert \,  \forall g\in K: gx=x\}$ and let $n_K$ be the dimension of $(\CC^n)^K$. The restriction $f_{\vert (\CC^n)^K}$ will be denoted by $f^K$. 
If $K$ is the cyclic subgroup $\langle g\rangle$ generated by an element $g\in G$,
we shall use the notations $(\CC^n)^g$, $n_g$ and $f^g$.

The Milnor fibre $V_f$ of the germ $f$ is
$f^{-1}(\varepsilon)\cap B_{\delta}^{2n}$ where $0<\vert\varepsilon\vert\ll\delta$ are small enough,
$B_{\delta}^{2n}$ is the ball of radius $\delta$ centred at the origin in $\CC^n$.
The group $G$ acts on the Milnor fibre $V_f$ and thus on its homology and cohomology groups.

\begin{definition} (cf.\ \cite{Ruan_etal})
 The {\em quantum cohomology group} of the pair $(f,G)$ is 
\begin{equation}\label{quantum_cohomology}
\calH_{f,G}= \bigoplus_{g\in G} \calH_g\,,
\end{equation}
where $\calH_g:=H^{n_g-1}(V_{f^g};\CC)^G= H^{n_g-1}(V_{f^g}/G;\CC)$ is the $G$-invariant part of the
vanishing cohomology group $H^{n_g-1}(V_{f^g};\CC)$ of the Milnor fibre
of the restriction of $f$ to the fixed point set of $g$. If $n_g=1$, this means the cohomology group  $\widetilde{H}^0(V_{f^g};\CC)$ reduced modulo a point. (We keep the notations without the tilde not to overload them.) If $n_g=0$, we assume $H^{-1}(V_{f^g};\CC)$ to be one dimensional  with the trivial action of $G$. This means that we consider the ``critical point'' of the function of zero variables to be non-degenerate and thus to have Milnor number equal to one and this corresponds to the definition of $\calH_{f,G}$ in the form given in \cite{Ruan_etal}.
\end{definition}

\begin{remarks}
 {\bf 1.} One can show that the restriction $f^g$ of $f$ to $(\CC^n)^g$ has an isolated critical
 point at the origin. Therefore its Milnor fibre is homotopy equivalent to a bouquet of
 spheres of dimension $(n_g-1)$.
 
 \noindent {\bf 2.} For a germ $f:(\CC^n,0)\to (\CC,0)$, let ${\rm Re\,}f:(\CC^n,0)\to (\RR,0)$ be its real
 part and let $V_{{\rm Re\,}f}=({\rm Re\,}f)^{-1}(\varepsilon)\cap B_{\delta}^{2n}$
 ($0<\varepsilon\ll\delta$) be its ``real Milnor fibre''.
 In~\cite{Ruan_etal} the space $\calH_g$ is defined as
 $\calH_g:=H^{n_g}(B_{\delta}^{2n},V_{{\rm Re\,}f^g};\CC)^G$. However this space
 is canonically isomorphic to $H^{n_g-1}(V_{f^g};\CC)^G$ (with the conventions for $n_g=0,1$ in the
 definition above).
 
 \noindent {\bf 3.} One can consider the {\em quantum homology group} of the pair $(f,G)$ defined as 
\begin{equation}\label{quantum_homology}
 \bigoplus_{g\in G} H_{n_g-1}(V_{f^g};\CC)^G\,,
\end{equation}
 where each summand on the right hand side is the $G$-invariant part of the
 (middle) homology group $H_{n_g-1}(V_{f^g};\CC)^G= H_{n_g-1}(V_{f^g}/G;\CC)$ of the Milnor fibre
 of the restriction of $f$ to the fixed point set of $g$. The majority of the constructions
 below are valid both for the quantum cohomology group and for the homology one. 
 \end{remarks}

Our aim is to define orbifold versions of the Milnor lattice and of the intersection form.
In singularity theory the intersection form is traditionally considered on the vanishing
homology group. Therefore below we shall mostly consider the quantum homology group
using the same notations $\calH_{f,G}$ and $\calH_g$. (In fact in \cite{Ruan_etal}
the description of the quantum cohomology group and of some structure on it starts
from the discussion of the corresponding homology group.)

\section{Orbifold monodromy operator} \label{sect:mono}
For a germ $f:(\CC^n,0)\to(\CC,0)$ with an isolated critical point at the origin the (classical)
monodromy transformation is a map $\varphi_f$ from the Milnor fibre
$V_f$ to itself
induced by rotating the (noncritical) value $\varepsilon$ around zero counterclockwise (see, e.g.,
\cite{AGV2}). If the function $f$ is quasihomogeneous, i.e., if there exist positive integers
$w_1$, \dots, $w_n$ and $d$ such that
$f(\lambda^{w_1}z_1, \ldots, \lambda^{w_n}z_n)=\lambda^df(z_1, \ldots, z_n)$
for $\lambda\in\CC$, this transformation can be defined by
\[
\varphi_f(z_1, \ldots, z_n)=(\exp(2\pi i \cdot w_1/d)z_1, \ldots , \exp(2\pi i \cdot w_n/d)z_n).
\]
This transformation is an element of the symmetry group of the function $f$. In the
FJRW theory it is usually denoted by $J$. The fix point set $(\CC)^J$ of the element $J$ is zero-dimensional.
Therefore the corresponding summand in (\ref{quantum_cohomology}) is one-dimensional. A generator of this
summand represents the unit element of the corresponding cohomological field theory.
The action of the monodromy transformation on the vanishing homology group of the singularity
$f$ is called the (classical) monodromy operator and will be denoted by $\varphi_f$ as well.

If the germ $f$ is invariant with respect to an action of a finite abelian group $G$ on $\CC^n$,
the classical monodromy transformation $\varphi_f$ can be assumed to be $G$-equivariant.
This implies that it preserves the fixed point sets of subgroups of $G$ in the Milnor fibre,
i.e., for a subgroup $K\subset G$ (in particular, for $K=\langle g\rangle$, $g\in G$), the
map $\varphi_f$ sends $V_{f^K}=V_f\cap (\CC^n)^K$ to itself and also induces a map
$\widehat{\varphi}_{f^K}$ from the quotient space $V_{f^K}/G$ to itself. The actions of these
maps on the homology groups $\calH_g$ define a map from the quantum homology group
$\calH_{f,G}$ to itself. 

The orbifold monodromy zeta function of
a pair $(f,G)$ consisting of a germ of a function $(\CC^n,0)\to(\CC,0)$ (not necessarily
non-degenerate, i.e.\ with an isolated critical point at the origin) and a finite group $G$ of its symmetries
(not necessarily abelian)
was defined in \cite{EG-Edinburgh}. 
We recall the definition for an abelian group $G$. The usual monodromy zeta function
of the transformation ${\widehat{\varphi}}_{f^{\langle g\rangle}}$ is defined by
\begin{equation*}
\zeta_{{\widehat{\varphi}}_{f^{\langle g\rangle}}}(t)=\prod\limits_{q\ge 0} 
\left(\det({\rm id}-t\cdot {\widehat{\varphi}}_{f^{\langle g\rangle}}^{*}
{\rm \raisebox{-0.5ex}{$\vert$}}{}_{H^q_c(V_{f^{\langle g\rangle}}/G;\RR)})\right)^{(-1)^q}\,
\end{equation*}
The definition of the orbifold monodromy zeta function is inspired by the notion of the orbifold spectrum (see, e.g., \cite{BH2, ET}). For an
element $g \in G$ acting on $\CC^n$ by
\[ g(z_1, \ldots , z_n)=(\exp(2 \pi i \cdot r_1)z_1, \ldots , \exp(2\pi i \cdot r_n)z_n)
\]
where $0 \leq r_j < 1$, $j=1, \ldots , n$, its {\em age} (or fermion shift number) \cite{Ito-Reid, Zaslow} is 
\[ {\rm age}(g) = \sum_{j=1}^n r_j \in \QQ_{\geq 0}.
\]
The map $g \mapsto \exp(2\pi i \cdot {\rm age}(g))$ defines a character
$\alpha_{\rm age} \in G^\ast={\rm Hom}(G,\CC^\ast)$. 
For an abelian group $G$, the orbifold monodromy zeta function $ \zeta^{{\rm orb}}_{f,G}(t)$ is given by the following equation
 \begin{equation*}
  \zeta^{{\rm orb}}_{f,G}(t)= \prod\limits_{g\in G}
  \left(\zeta_{{\widehat{\varphi}}_{f^{\langle g\rangle}}}(\exp(-2\pi i\,{\rm age\,}(g))t)\right)\,.
 \end{equation*}
The reduced orbifold monodromy zeta function $\overline{\zeta}^{{\rm orb}}_{f,G}(t)$ is defined by
$$
  \overline{\zeta}^{{\rm orb}}_{f,G}(t)=
  \zeta^{{\rm orb}}_{f,G}(t)\left/\prod\limits_{g\in G}(1-\exp(-2\pi i\,{\rm age\,}(g))t) \right..
$$

It was shown that the reduced orbifold monodromy zeta functions of
Berglund--H\"ubsch--Henningson dual pairs (not necessarily non-degenerate ones) either coincide or are inverse
to each other depending on the number $n$ of variables. If the function $f$ has an isolated critical point
at the origin, the reduced monodromy zeta function coincides with the {\em characteristic polynomial} of
the pair $(f,G)$ defined in \cite{ET} (which is not always a polynomial).

The definition of the orbifold monodromy zeta function leads to the following definition.

\begin{definition} The {\em orbifold monodromy operator} $\varphi_{f,G}$ on the quantum (co)homology group
$\calH_{f,G}$ is the direct sum of the operators $\alpha_{\rm age}(g) \cdot \widehat{\varphi}_{f^g}$ on
$\calH_g$.
\end{definition}

One can show that the reduced monodromy zeta function $\overline{\zeta}^{\rm orb}_{f,G}(t)$ coincides with
the zeta function of the orbifold monodromy operator $\varphi_{f,G}$ .

The orbifold monodromy operator does not preserve the natural lattice $H_{n_g-1}(V_{f^g};\ZZ)^G$
in $\calH_g$ (and, in general, no lattice in it at all). A lattice in $\calH_{f,G}$ preserved by
the orbifold monodromy operator $\varphi_{f,G}$ will be described below.

\section{Group rings} \label{sect:rings}
Let $R$ be a commutative ring with unity (usually either the field $\CC$ of complex numbers
or the ring $\ZZ$ of integers).
For a finite abelian group $K$, let $R[K]$ be the corresponding group ring.
It is a free $R$-module with the basis $\{e_g\}$ whose elements correspond to the elements $g$ of the
group $K$. Let $K^*={\rm Hom\,}(K,\CC^*)$ be the group of characters of the group $K$.
As an abstract group $K^*$ is isomorphic to $K$, but not in a canonical way.

The space $\CC[K]$ carries a natural representation of the group $K$ defined by $he_g=e_{hg}$ for
$h\in K$. A change of the basis permits to identify $\CC[K]$ as a vector space (not as a ring) with the vector space
$\CC[K^*]$. Namely, one should define the new basis $\e_{\alpha}$, $\alpha\in K^*$, by
\begin{equation}\label{change}
 \e_{\alpha}:=\sum_{g\in K}\left(\alpha(g)\right)^{-1}e_g\,.
\end{equation}
In the other direction one has
\begin{equation}\label{change_back}
e_g=\frac{1}{\vert K\vert}\sum_{\alpha\in K^*}\left(\alpha(g)\right)\e_{\alpha}\,.
\end{equation}

For a character $\beta:K\to\CC^*$, let $\psi_{\beta}$ be the linear map from $\CC[K]$ to $\CC[K]$
defined by
$$
\psi_{\beta}(e_g)=\beta(g)e_g\,.
$$
One has
$$
\psi_{\beta}(\e_{\alpha})=\sum_{g\in K}\left(\alpha(g)\right)^{-1}\beta(g)e_g=
\sum_{g\in K}\left(\alpha(g)(\beta(g))^{-1}\right)^{-1}e_g=\e_{\alpha\beta^{-1}}\,.
$$
This implies that the map $\psi_{\beta}$ preserves the lattice $\ZZ[K^*]\subset\CC[K^*]$.

For a subgroup $H$ of $K$ one has a natural map from $K^*$ to $H^*$: the restriction of characters.
This map is epimorphic. It induces a ring epimorphism $r^K_H:\CC[K^*]\to \CC[H^*]$.
\begin{lemma} \label{lem:Ker}
 The kernel ${\rm Ker\,} r^K_H$ of the homomorphism $r^K_H$ coincides with the subspace of $\CC[K]$ generated by the basis
 elements $e_g$ with $g\in K\setminus H$.
\end{lemma}\label{lemma1}

\begin{proof}
 Let $A_H$ be the kernel of the natural map $K^*\to H^*$. From (\ref{change_back}) one has
 $$
 r^K_H e_g=\frac{1}{\vert K\vert}\sum_{\beta\in H^*}
 \left(\sum_{\alpha\in K^*:
 \alpha_{\vert H}=\beta}\alpha(g)\right)\widehat{e}_\beta=
 \frac{1}{\vert K\vert}\sum_{\beta\in H^*}
 \left(\widehat{\beta}(g) \sum_{\alpha\in A_H}\alpha(g)\right)\widehat{e}_\beta\,,
 $$
 where $\widehat{\beta}$ is an element of $K^*$ such that $\widehat{\beta}_{\vert H}=\beta$.
 The element $g$ as an element of $\left(K^*\right)^*=K$ defines a non-trivial character on
 the subgroup $A_H$. Therefore $\sum\limits_{\alpha\in A_H}\alpha(g)=0$ and thus $r^K_H e_g=0$.
 This shows that $\langle e_g:g\in K \setminus H\rangle\subset {\rm Ker\,}r^K_H$. On the other hand the dimensions of
 these spaces are equal to $\vert K\vert-\vert H\vert$.
\end{proof}

\begin{remark}
 If $K$ is the maximal group $G_f$ of diagonal symmetries of an invertible polynomial $f$ and
 $H$ ia a subgroup of $K$, then the dual group $\widetilde{H}$ in the Berglund--H\"ubsch--Henningson dual
 pair $(\widetilde{f},\widetilde{H})$ is the kernel of the map $K^*\to H^*$ indicated above.
 Pay attention that ${\rm Ker\,} r^K_H$ is not isomorphic to $\CC[\widetilde{H}]$. (In particular,
 the dimension of the latter one is equal to $\vert K\vert/\vert H\vert$.)
\end{remark}

The map $\psi_{\beta}$ described above preserves the lattice
$\ZZ[K^*]\cap {\rm Ker\,} r^K_H\subset {\rm Ker\,} r^K_H$.

\section{Orbifold Milnor lattice} \label{sect:Milnor}
Here we define a lattice in the quantum homology group $\calH_{f,G}$
which can be considered as an orbifold version of the Milnor lattice.

Let $G$ be a finite abelian group acting faithfully on $(\CC^n,0)$. Let $f:(\CC^n,0)\to(\CC,0)$ be a germ of a $G$-invariant
holomorphic function with an isolated critical point at the origin.
For a point $x\in \CC^n$, let
$G_x=\{g\in G\, | \, gx=x\}$ be the isotropy subgroup of $x$. One has $(\CC^n)^K=\{x\in \CC^n \,|\,
G_x\supset K\}$. Let ${\rm Iso\,}G$ be the set of the subgroups of $G$ which are isotropy subgroups
of some points (i.e., $K\in {\rm Iso\,}G$ iff $\exists x\in\CC^n:G_x=K$).

Let $K$ be a subgroup of $G$ belonging to ${\rm Iso\,}G$. Let $E_K\subset \CC[K^*]$ be the
intersection of the kernels of the maps $r^K_H$ for all $H \in {\rm Iso\,}G$ such that
$H\varsubsetneq K$. Lemma~\ref{lem:Ker} implies that the subspace $E_K$ is generated by
all the basis elements $e_g$ of $\CC[K]$ with
$$
g\in {{\stackrel{\circ}{K}}}:=
K\setminus\bigcup\limits_{\begin{array}{c}H\in{\rm Iso\,}G,\atop H\varsubsetneq K\end{array}}H\,.
$$

There is a natural lattice $\ZZ[K^*]$ in $\CC[K^*]$. Its intersection with the subspace $E_K$
gives a lattice there. Using it one can define a lattice in the quantum homology group
$\calH_{f,G}$ in the way described below. The definition of the lattice $\ZZ[K^*]\cap E_K$
(and thus of the corresponding lattice in $\calH_{f,G}$) does not take into account the ages
of the elements of the group $G$ although they constitute an important part of the quantum
singularity theory. Moreover symmetric bilinear forms on $\calH_{f,G}$ which can be constructed
in a somewhat natural way and which could be considered as orbifold analogues of the (symmetric)
intersection form on the vanishing homology group of a singularity appear to be not integral
or at least not even on the corresponding lattice.
Therefore we consider another lattice in $\CC[K^*]$ and thus in $\calH_{f,G}$.

Let $\alpha_K\in K^*$ be the restriction of $\alpha_{\rm age}\in G^*$ to $K$ (in particular,
$\alpha_G=\alpha_{\rm age}$) and let $p_K$ be the order of $\alpha_K$. Let $\ZZ^{(p_K)}[K^*]$
be the sublattice of $\ZZ[K^*]$ defined by
$$
\ZZ^{(p_K)}[K^*]=\left\{\sum\limits_{\alpha\in K^*}m_{\alpha}\widehat{e}_{\alpha}\, \left| \,
\forall \beta\in K^*:\ \sum_{j=1}^{p_K} m_{\beta\alpha_K^j}
\ \ {\text{is divisible by}}\ \ p_K\right.\right\}\,.
$$
(In other words the sublattice $\ZZ^{(p_K)}[K^*]$ is the kernel of the natural map
$\ZZ[K^*]\to\ZZ_{p_K}[K^*]$.)
Let $E_K^{\ZZ}:=E_K\cap \ZZ^{(p_K)}[K^*]$.

Let
$$
\calH_K:=H_{n_K-1}(V_{f^K};\CC)^G\,.
$$
The space $\calH_K$ contains the natural lattice
$\calH_K^{\ZZ}=H_{n_K-1}(V_{f^K};\ZZ)^G$.

By definition
$$
\calH_{f,G}=\bigoplus_{g\in G} \calH_g\,.
$$
All the summands on the right hand side are of the form $\calH_K$ for $K\in {\rm Iso\,}G$.
The space $\calH_K$ appears as the summand $\calH_g$ if and only if $g\in{\stackrel{\circ}{K}}$.
Thus one has
$$
\calH_{f,G}=\bigoplus_{K\in {\rm Iso\,}G}\bigoplus_{g\in {\stackrel{\circ}{K}}} \calH_K\,.
$$
Therefore
\begin{equation} \label{eq:qhg}
\calH_{f,G}=\bigoplus_{K\in {\rm Iso\,}G}E_K\otimes \calH_K\,.
\end{equation}
The tensor product
of two complex vector spaces with distinguished lattices contains a natural lattice as well.
Therefore the quantum homology group $\calH_{f,G}$ contains the natural lattice
$\Lambda_{f,G}=\bigoplus\limits_{K\in {\rm Iso\,}G} E_K^{\ZZ}\otimes \calH_K^{\ZZ}$.

Let us recall that $\psi_{\alpha_{K}}$ is a map from $\CC[K]$ to itself sending the basis
element $e_g$ to $\alpha_{K}(g)e_g$. The orbifold monodromy operator is
$$
\varphi_{f,G} = \bigoplus_{K\in {\rm Iso\,}G}\psi_{\alpha_{K}}\otimes \widehat{\varphi}_{f^K},
$$
where $\widehat{\varphi}_{f^K}$ is the map $\calH_K\to\calH_K$ induced by the classical monodromy operator.
Since $(\psi_{\alpha_{K}})_{\vert E_K}$ and $\widehat{\varphi}_{f^K}$ preserve the corresponding lattices
in $E_K$ and in $\calH_K$, the orbifold monodromy operator $\varphi_{f,G}$ preserves the lattice
$\Lambda_{f, G}$. Thus we have proved the following statement.

\begin{theorem}\label{theo-lattice}
 There exists a well defined lattice $\Lambda_{f, G}$ in the quantum homology group
 $\calH_{f, G}$ invariant with respect to the orbifold monodromy operator $\varphi_{f,G}$.
\end{theorem}

\begin{definition}
The lattice $\Lambda_{f, G}$ in  $\calH_{f, G}$ will be called the {\em orbifold Milnor lattice}
of the pair $(f,G)$.
\end{definition}

\section{Orbifold Seifert form} \label{sect:Seifert}
An essential aim of this paper is to define an analogue of the intersection form on the
quantum homology group (or on the orbifold Milnor lattice). To describe properties of
the intersection form and its relations with the monodromy transformation, it is useful
to use the Seifert form (or the so-called variation operator), see, e.g., \cite{AGV2}.

For a germ $f:(\CC^n,0) \to (\CC,0)$ of a holomorphic function with an isolated critical point
at the origin, the Seifert form is a (non-degenerate) bilinear form on the vanishing homology
group (the Milnor lattice) $H_{n-1}(V_f;\ZZ)$ ($V_f$ is the Milnor fibre of $f$), or, in other words,
a linear map $L:H_{n-1}(V_f;\ZZ) \to (H_{n-1}(V_f;\ZZ))^\ast$ (see, e.g., \cite{AGV2}).
In general, this form is neither symmetric nor skew-symmetric. (The group $(H_{n-1}(V_f;\ZZ))^\ast$
dual to $H_{n-1}(V_f;\ZZ)$ is isomorphic to the relative homology group
$H_{n-1}(V_f, \partial V_f;\ZZ)$ or to the cohomology group $H^{n-1}(V_f;\ZZ)$.)
In a so-called distinguished basis of the Milnor lattice $H_{n-1}(V_f;\ZZ)$ (and the dual
basis of $(H_{n-1}(V_f;\ZZ))^\ast$), the matrix of the operator $L$ is an upper triangular matrix with
the diagonal elements equal to $(-1)^{\frac{n(n+1)}{2}}$. The intersection form $S$ on
$H_{n-1}(V_f;\ZZ)$ is equal to 
\begin{equation} \label{eq:S}
S=-L+ (-1)^n L^T,
\end{equation}
where $L^T$ is the transposed form. The monodromy operator
$\varphi_f : H_{n-1}(V_f;\ZZ) \to H_{n-1}(V_f;\ZZ)$ is given by the equation
\begin{equation} \label{eq:mon}
\varphi_f = (-1)^n L^{-1}L^T.
\end{equation} 

An important advantage of the Seifert form (compared with the intersection form) is the formula
for the Seifert form of the Sebastiani-Thom (``direct'') sum of singularities.
If $f_1:(\CC^m,0) \to (\CC,0)$ and $f_2:(\CC^n,0) \to (\CC,0)$ are two germs with isolated
critical points at the origin then the vanishing homology group $H_{m+n-1}(V_{f_1 \oplus f_2}; \ZZ)$
of the germ $f_1 \oplus f_2: (\CC^{m+n},0) \to (\CC,0)$,
$(f_1 \oplus f_2)(\underline{x}, \underline{y}) = f_1(\underline{x})+f_2(\underline{y})$,
is canonically isomorphic to the tensor product $H_{m-1}(V_{f_1};\ZZ) \otimes H_{n-1}(V_{f_2};\ZZ)$
of the corresponding vanishing homology groups. One has (see, e.g., \cite[Theorem~2.10]{AGV2})
\begin{equation} \label{eq:Thom}
L_{f_1 \oplus f_2} = (-1)^{mn} L_{f_1} \otimes L_{f_2}.
\end{equation}
A translation of this property in terms of the intersection form can be formulated in a
reasonable way only in a distinguished basis and is given by somewhat involved formulae.

If a finite group $G$ acts on $(\CC^n,0)$ and $f: (\CC^n,0) \to (\CC,0)$ is a $G$-invariant
germ with an isolated critical point at the origin, the monodromy operator preserves the
$G$-invariant part of the vanishing homology group, the restriction $L^G$ of the Seifert form
to the $G$-invariant part is non-degenerate and it is related with the restriction
of the monodromy operator and of the intersection form by the same equations as (\ref{eq:S})
and (\ref{eq:mon}). If $f_1$ and $f_2$ are two germs as above and $f_1$ is $G$-invariant
with respect to an action of a finite group $G$ on $\CC^m$, then $f_1 \oplus f_2$ is
$G$-invariant with respect to the $G$-action on $\CC^m \oplus \CC^n$ which is the trivial
extension of the one on $\CC^m$. One has
$H_{m+n-1}(V_{f_1 \oplus f_2}; \ZZ)^G \simeq
H_{m-1}(V_{f_1};\ZZ)^G \otimes H_{n-1}(V_{f_2};\ZZ)$ and
\begin{equation} \label{eq:Thom_equi}
 L^G_{f_1 \oplus f_2} = (-1)^{mn} L^G_{f_1} \otimes L_{f_2}\,.
\end{equation}


Here we define a ``Seifert form'' (an integer valued bilinear form) on the orbifold
Milnor lattice $\Lambda_{f,G}$. For that we shall identify $\CC[K^*]\otimes \calH_K$,
$K\in {\rm Iso\,}G$, with a direct sum of vanishing homology groups of certain singularities
so that the orbifold monodromy operator on it becomes the direct sum of (the restrictions of)
the corresponding classical monodromy operators. The direct sum of the corresponding
Seifert forms on the summands
gives a bilinear form on $\CC[K^*]\otimes \calH_K$ with integer values on the lattice
$\ZZ^{(p_K)}[K^*]\otimes \calH_K^{\ZZ}$.
The direct sum over all subgroups
$K\in {\rm Iso\,}G$ of the restrictions of these forms to $E_K\otimes \calH_K$ gives a bilinear
form on $\calH_{f,G}$ with integer values on the orbifold Milnor lattice $\Lambda_{f,G}$ which will be
called the orbifold Seifert form.

The space $\CC[K^*]$ has a decomposition into parts corresponding to the orbits
of the multiplication by $\alpha_K$, i.e., to the elements of $K^*/\langle\alpha_K\rangle$:
$$
\CC[K^*]=\bigoplus_{\gamma\in K^*/\langle\alpha_K\rangle}\langle\widehat{e}_{\alpha}:
[\alpha]=\gamma\rangle\,.
$$
Each summand $B_{\gamma}=\langle\widehat{e}_{\alpha}: [\alpha]=\gamma\rangle$ on the right hand side
can be represented as the direct sum of two subspaces:
$B_{\gamma,1}\cong\CC$ generated by the element $\sum\limits_{[\alpha]=\gamma}\widehat{e}_{\alpha}$
and $B_{\gamma,2}$ consisting of the elements
$\sum\limits_{[\alpha]=\gamma}c_{\alpha}\widehat{e}_{\alpha}$ such that $\sum c_{\alpha}=0$.
The subspaces $B_{\gamma,1}\cong\CC$ and $B_{\gamma,2}$ contain natural lattices:
$\langle\sum\limits_{[\alpha]=\gamma}\widehat{e}_{\alpha}\rangle$ and
$\ZZ^{(p_K)}[K^*]\cap B_{\gamma,2}=\ZZ[K^*]\cap B_{\gamma,2}$ respectively.

If $p_k=1$, then $B_{\gamma,2}=0$. If $p_k \geq 2$, then the space $B_{\gamma,2}$ with the
lattice $\ZZ^{(p_K)}[K^*]\cap B_{\gamma,2}=\ZZ[K^*]\cap B_{\gamma,2}$
in it can be identified with the vanishing homology group of the $A_{p_K-1}$-singularity
with the Milnor lattice in it in the following way.
The $A_{p_K-1}$-singularity is defined by the function $u^{p_K}$. The Milnor fibre
$V_{u^{p_K}}=\{u^{p_K}=1\}$ of it consists of the points $u_j=\exp{(-2\pi i (j-1)/p_K)}$,
$j=1,\ldots,p_K$. A (distinguished) basis of the Milnor lattice in the vanishing homology
group $\widetilde{H}_0(V_{u^{p_K}})$ is formed by the vanishing cycles
$\Delta_1=u_1-u_2$, $\Delta_2=u_2-u_3$, \dots, $\Delta_{p_K-1}=u_{p_K-1}-u_{p_K}$
(see, e.g., \cite[Theorem 2.5]{AGV2}, where one has a small misprint (a wrong sign)
in the formula for $u_j$; in the initial Russian version the sign is correct).
Sometimes it is convenient to consider also the vanishing cycle $\Delta_{p_K}=u_{p_K}-u_{1}$
keeping in mind that $\sum\limits_{j=1}^{p_k}\Delta_{j}=0$.
The Seifert operator $L$ is defined by $L\Delta_1=-\Delta_1^*$,
$L\Delta_j=-\Delta_j^*+\Delta_{j-1}^*$ for $j>1$, where $\Delta_1^*$, $\Delta_2^*$, \dots, $\Delta_{p_K-1}^*$
is the dual basis in the dual lattice. (The monodromy transformation of the function $u^{p_K}$
permutes the points $u_j$ cyclically sending $u_j$ to $u_{j-1}$ and therefore sending
the cycle $\Delta_j$ to $\Delta_{j-1}$ for $j>1$ and the cycle $\Delta_1$ to
$\Delta_{p_K}=-\sum\limits_{j=1}^{p_K-1}\Delta_j$.) Let us consider the following (integer) basis of
$B_{\gamma,2}$. Let $\alpha\in K^*$ be an arbitrary representative of $\gamma$.
Then the set $\delta_1=\widehat{e}_{\alpha_K\alpha}-\widehat{e}_{\alpha}$,
$\delta_2=\widehat{e}_{\alpha_K^2\alpha}-\widehat{e}_{\alpha_K\alpha}$, \dots,
$\delta_{p_K-1}=\widehat{e}_{\alpha_K^{p_K-1}\alpha}-\widehat{e}_{\alpha_K^{p_K-2}\alpha}$
is a basis of $\ZZ^{(p_K)}[K^*]\cap B_{\gamma,2}$. We identify $B_{\gamma,2}$ with the
vanishing homology group $\widetilde{H}_0(V_{u^{p_K}})$ of the $A_{p_K-1}$-singularity
by identifying the vanishing cycles $\Delta_j$ with the basis elements $\delta_j$.
This identification respects the corresponding lattices.

\begin{remark}
 It is important to note that the bilinear form on $B_{\gamma,2}$ defined through this
 identification by the Seifert form on the $A_{p_K-1}$-singularity and therefore also analogues
 of the Seifert forms defined below do not depend on the choice of the representative
 $\alpha$ in an orbit $\gamma$.
\end{remark}

To define an analogue $\widehat{L}_K$ of the Seifert form on $\CC[K^*]\otimes\calH_K$,
we shall define an analogue $\widehat{\ell}_K$ of the Seifert form on $\CC[K^*]$.
Let $\widehat{\ell}_K$ be the direct sum of the Seifert forms $\ell_{\gamma,2}=L_{u^{p_K}}$ of the
$A_{p_K-1}$-singularity on $B_{\gamma, 2}$ and the form $\ell_{\gamma,1}$ on $B_{\gamma,1}\cong\CC$
defined by
$$
\ell_{\gamma,1}\left(\sum\limits_{\alpha:[\alpha]=\gamma}\widehat{e}_{\alpha},
\sum\limits_{\alpha:[\alpha]=\gamma}\widehat{e}_{\alpha} \right)
=-1\,.
$$
Let $I:\CC[K^*]\to\CC[K^*]$ be the operator which is the identity on $B_{\gamma, 2}$ and minus
the identity on $B_{\gamma, 1}$ for each $\gamma$.

\begin{proposition}
 The bilinear form $\widehat{\ell}_K$ is integer valued on the lattice $E_K^{\ZZ}$ and possesses the property
 \begin{equation}\label{psi_alpha}
  \psi_{\alpha_K}=-I \, \widehat{\ell}_K^{-1} \widehat{\ell}_K^T\,.
 \end{equation}
\end{proposition}

\begin{proof}
 The lattice $\ZZ^{(p_K)}[K^\ast]$ is the direct sum of the lattices $\ZZ[K^*]\cap B_{\gamma,2}$ and
 $\ZZ[K^*]\cap B_{\gamma,1}$. The form $\widehat{\ell}_K$ is integer valued on them and they are orthogonal
 to each other with respect to $\widehat{\ell}_K$. Therefore $\widehat{\ell}_K$ is integer valued on
 $\ZZ^{(p_K)}[K^*]\supset E_K^{\ZZ}$.
 
 The restriction of the operator $\psi_{\alpha_K}$ to $B_{\gamma,2}$ coincides with the
 monodromy operator of the $A_{p_K-1}$-singularity and therefore is equal to $-\widehat{\ell}_K^{-1}\widehat{\ell}_K^T$.
 The restriction of $\psi_{\alpha_K}$ to the (one-dimensional) space $B_{\gamma,1}$ is the identity and thus coincides with
 $\widehat{\ell}_K^{-1}\widehat{\ell}_K^T$.
\end{proof}

\begin{remark}
 The presence of the operator $I$ in Equation~(\ref{psi_alpha}) is related with the fact
 that, in some sense, $B_{\gamma,2}$ and $B_{\gamma,1}$ are identified with the vanishing
 homology groups of functions with different numbers of variables: $1$ and $0\mod 4$
 respectively. Stabilization of a function means addition of the sum of squares
 of new variables. If one of two functions is right-equivalent to the stabilization of the other one,
 the functions are called stably equivalent. Topological properties of stably equivalent functions
 are closely related. Their Milnor lattices can be identified. Via these identification, the Seifert forms
 of stably equivalent functions may differ, but only by a sign. Moreover, they are 4-periodic: if the numbers of
 variables of functions have the same residue modulo 4, their Seifert forms coincide. The monodromy
 transformations of stably equivalent functions are 2-periodic. Thus from the topological point of view
 only the residue of the number of variables modulo 4 matters.
 The orbifold monodromy operator acts on $B_{\gamma,2}$ and on $B_{\gamma,1}$ as on the vanishing homology
 group of the $A_{p_K-1}$-singularity of 1 variable and the $A_1$-singularity of 0 variables respectively.
 
\end{remark}

\begin{theorem}\label{orb-Z}
 The restriction $\ell_K$ of the bilinear form $\widehat{\ell}_K$ to $E_K$ is non-degenerate
 and possesses the property 
 \begin{equation}
  \psi_{\alpha_K}=-I \, \ell_K^{-1}\ell_K^T.
 \end{equation}
\end{theorem}

\begin{proof}
 Let $H\in {\rm Iso\,}G$, $H\varsubsetneq K$, and let $A_H$ be the kernel of the natural map
 $K^*\to H^*$. The subgroup $A_H\subset K^*$ acts on $\CC[K^*]$ by permutations of the basis
 elements $\widehat{e}_{\alpha}$. Let $\CC[K^*]=\CC[K^*]^{A_H}\oplus\CC[K^*]^{\overline{A_H}}$
 be the decomposition of $\CC[K^*]$ into the invariant part $\CC[K^*]^{A_H}$ and the
 ``non-invariant part'' $\CC[K^*]^{\overline{A_H}}$, i.e., the sum of all the parts corresponding
 to non-trivial representations of $A_H$. For $\beta\in H^*$, the group $A_H$ acts on the subspace
 $\langle\widehat{e}_{\alpha}: \alpha_{\vert H}=\beta\rangle\subset \CC[K^*]$ by the regular
 representation. This representation is the direct sum of the (one-dimensional) invariant part
 generated by the element $\sum\limits_{\alpha:\alpha_{\vert H}=\beta}\widehat{e}_{\alpha}$
 and the non-invariant part consisting of all the linear combinations of the elements 
 $\widehat{e}_{\alpha}$, $\alpha_{\vert H}=\beta$,  with the sum of the coefficients equal to zero. The non-invariant part
 coincides with the kernel of the restriction to
 $\langle\widehat{e}_{\alpha}: \alpha_{\vert H}=\beta\rangle$ of the map $r^K_H:\CC[K^*]\to\CC[H^*]$.
 This implies that ${\rm Ker\,}r^K_H=\CC[K^*]^{\overline{A_H}}$.
 
 For all subgroups $H\in {\rm Iso\,}G$ such that $H\varsubsetneq K$, the actions of the subgroups
 $A_H\subset K^*$ commute and the space $\CC[K^*]$ decomposes into the parts corresponding to
 different representations of these subgroups. The subspace $E_K$ is the intersection of the
 subspaces $\CC[K^*]^{\overline{A_H}}$ for all $H\in {\rm Iso\,}G$, $H\varsubsetneq K$.
 This means that it is the (direct) sum of all the parts on which the representations
 of all the groups $A_H$ are non-trivial.
 The operator $\psi_{\alpha_K}$ commutes with these actions and the bilinear form
 $\widehat{\ell}_K$ is invariant with respect to them. This implies that these parts are
 orthogonal to each other with respect to the bilinear form $\widehat{\ell}_K$, the restriction
 of $\widehat{\ell}_K$ to each of these parts is non-degenerate and satisfies the relation
 (\ref{psi_alpha}). This implies the statement. 
\end{proof}

\begin{definition}
 The Seifert form $\widehat{L}_K$ on $\CC[K^*]\otimes\calH_K$ is
 $$
 \widehat{L}_K=(-1)^{n_K}\widehat{\ell}_K\otimes L^G_{f^K}
 $$
 (cf.~(\ref{eq:Thom_equi}): the Seifert form $\widehat{L}_K$ is defined as the Seifert form of
 the Sebastiani-Thom sum of functions of one variable and of $n_K$ variables respectively).
\end{definition}

The representation of $\CC[K^*]$ in the form
$$
\bigoplus_{\gamma\in K^*/\langle\alpha_K\rangle}\left(\widetilde{H}_0(V_{u^{p_K}};\CC)\oplus\CC\right)
$$
gives an isomorphism between $\CC[K^*]\otimes\calH_K$ and
\begin{eqnarray*}
& & \bigoplus_{\gamma\in K^*/\langle\alpha_K\rangle}\left(\widetilde{H}_0(V_{u^{p_K}};\CC)\otimes\calH_K\oplus\calH_K\right) \\
& = & \bigoplus_{\gamma\in K^*/\langle\alpha_K\rangle}\left(H_{n_K}(V_{f^K+u^{p_K}};\CC)^G\oplus H_{n_K-1}(V_{f^K};\CC)^G\right)\,.
\end{eqnarray*}
The form $\widehat{L}_K$ is the direct sum of the Seifert forms on the first summands and the Seifert forms on the second summands stabilized to
the same number of variables.

Equation~(\ref{psi_alpha}) implies that
\begin{equation}\label{L-mono}
\psi_{\alpha_K}\otimes \widehat{\varphi}_{f^K}=-I \, \widehat{L}_K^{-1}\widehat{L}_K^{T}.
\end{equation}

\begin{theorem}
The restriction $L_K$ of the form $\widehat{L}_K$ to $E_K\otimes\calH_K$ is a non-degenerate bilinear form such that
$$
(\psi_{\alpha_K}\otimes \widehat{\varphi}_{f^K})_{\vert E_K\otimes\calH_K}=-I \, L_K^{-1}L_K^{T}.
$$
\end{theorem}

\begin{proof}
This is a direct consequence of Equation~(\ref{L-mono}) and Theorem~\ref{orb-Z}.
\end{proof}

\begin{definition} The {\em orbifold Seifert form} $L_{f,G}$ on the quantum homology group $\calH_{f,G} = \bigoplus E_K \otimes \calH_K$ is the direct sum of the forms $L_K$ on $E_K \otimes \calH_K$.
\end{definition}

\section{Intersection forms on the orbifold Milnor lattice} \label{sect:inter}
Here we define bilinear forms on the quantum homology group which are analogues of the intersection forms
on the vanishing homology groups of singularities. The intersection form on the vanishing homology group
of a singularity is either symmetric or skew-symmetric depending on the number of variables. To each singularity
one also associates a symmetric form which is the intersection form of its stabilization with an odd number of
variables. (A tradition of singularity theory is to consider the stabilization with the number of variables
equal to $3\mod 4$.) The symmetric intersection form appears to be a more important invariant of singularities than the non-symmetric one.

For a germ $f:(\CC^n,0)\to(\CC,0)$ the intersection form $S(\cdot,\cdot)$ on the vanishing homology group
$H_{n-1}(V_f,\CC)$ is defined by the Seifert form. Namely, one has
$S=-L+(-1)^nL^T$. This inspires the following definition.

\begin{definition}
The {\em mixed intersection form} on the quantum homology group $\calH_{f,G}$ is defined by
$$
S^{\rm mix}_{f,G}=\bigoplus_{K\in {\rm Iso\,}G}\left(-L_K+(-1)^{n_K}L_K^T\right)\,.
$$
\end{definition}

It is an integer valued bilinear form on $\calH^{\ZZ}_{f,G}$ (symmetric or skew symmetric on the summands in (\ref{eq:qhg})).

Essentially (up to sign) there are two natural symmetric bilinear forms on the quantum homology group $\calH_{f,G}$.
One of them is obtained by the stabilization of each summand to a function of $3\mod 4$ variables. The other one
(more natural from our point of view) is obtained by the stabilizations to $3\mod 4$ variables for $n_K$ odd
and to $1\mod 4$ variables for $n_K$ even respectively (or vice versa).

\begin{definition}
The {\em orbifold intersection form} on the quantum homology group $\calH_{f,G}$ is
$$
S^{\rm orb}_{f,G}=\bigoplus_{K\in {\rm Iso\,}G}(-1)^{\frac{n_K(n_K+1)}{2}}\left(-L_K-L_K^T\right)\,.
$$
The {\em quantum intersection form} on $\calH_{f,G}$ is
$$
S^{\rm qua}_{f,G}=\bigoplus_{K\in {\rm Iso\,}G}(-1)^{\frac{(n_K-2)(n_K+1)}{2}}\left(-L_K-L_K^T\right)\,.
$$
\end{definition}

The reason for these names is the following. The quantum intersection form is ``predominantly negative''.
This means that it is induced by intersection forms of singularities with the self-intersection numbers
of the vanishing cycles equal to $(-2)$. The orbifold intersection form is ``predominantly negative''
on the summands with $n_K$ odd and is ``predominantly positive'' (i.e., induced by intersection forms 
with the self-intersection numbers of the vanishing cycles equal to $(+2)$) on the summands with $n_K$ even.
In the second case the signature of this form is more related to the orbifold Euler characteristic,
whence in the first case the signature is more related to the rank of the quantum homology group.

Summarizing the facts from Section~\ref{sect:Seifert}, we have

\begin{proposition}
The bilinear forms $S^{\rm mix}_{f,G}$, $S^{\rm orb}_{f,G}$ and $S^{\rm qua}_{f,G}$ are integer valued forms on
$\calH_{f,G}^{\ZZ}$. The forms $S^{\rm orb}_{f,G}$ and $S^{\rm qua}_{f,G}$ are symmetric and even.
\end{proposition}

\begin{remark}
 It is interesting to understand a relation of the defined bilinear forms with the bilinear pairing
 considered in the FJRW-theory: \cite[page 38]{Ruan_etal}. However, a direct relation is unclear.
 The pairing in the FJRW-theory is well defined only if the group $G$ contains the exponential
 grading operator $J$, whereas the definitions of the pairings introduced here do not require
 additional conditions on the group $G$. The pairing in the FJRW-theory is defined through pairings
 on the summands $\calH_g$ of the quantum (co)homology group \cite[Definition~3.1.1]{Ruan_etal}.
 In fact the pairing on $\calH_g$ is nothing else but the Seifert form of the germ $f^g$ restricted to
 the subspace of the vanishing
 homology group of $f^g$ invariant with respect to $G$. The Seifert form itself is not symmetric,
 however its restriction to the subspace invariant with respect to the classical monodromy operator $J$ is.
\end{remark}

\section{Invertible polynomials} \label{sect:inv}
In this section, we compute the orbifold Milnor lattice for some examples. These examples are chosen in the class of so called invertible polynomials.  An invertible polynomial  is a quasihomogeneous polynomial with the number of monomials equal to the number of variables. We consider an invertible polynomial $f$ and a subgroup $G$ of the group $G_f$ of diagonal symmetries of $f$. A description of properties of invertible polynomials and of their symmetry groups can be found, e.g., in \cite{Krawitz, KS, EGT}. In particular, for a pair $(f,G)$ consisting of a (non-degenerate) invertible polynomial $f$ and a subgroup $G$ of its symmetry group $G_f$, one can consider the (Berglund--H\"ubsch--Henningson) dual pair $(\widetilde{f}, \widetilde{G})$.  It is an interesting problem to compare the orbifold Milnor lattices for dual pairs. We shall examine some examples. 

The first result is that  the quantum cohomology groups of dual pairs have the same rank. This was shown in \cite[Theorem~1.1]{Krawitz} for a pair $(f,G)$, where $G$ is an admissible group, i.e., a group containing the exponential grading operator $J$. Here we give a proof for an arbitrary group $G$. For this purpose we adapt certain results of \cite{EGT}.


The quantum cohomology group $\calH_{f,G}$ is the direct sum of the subspaces $\calH_{f,G,0}$ and $\calH_{f,G,1}$ where
\[
\calH_{f,G,i} = \bigoplus_{g \in G, \atop n_g \equiv i \, {\rm mod} \, 2} \calH_g \quad \mbox{for } i=0,1.
\]
A ($\QQ \times \QQ$)-grading on these spaces was defined in \cite[Equations~(2.2),(2.3)]{EGT} in terms of the mixed Hodge structures on the vanishing cohomology groups and the ages of elements of $G$.
It is the same one as the bigrading considered in \cite[Remark 3.2.4]{Ruan_etal}.

For the next definition compare \cite[Equation~(2.4)]{EGT}.

\begin{definition}
The {\em E\,$^i$-function} of the pair $(f,G)$, $i=0,1$, is
\begin{equation}
E^i(f,G)(t,\bar{t}):=\displaystyle\sum_{p,q\in\QQ} {\rm dim}_\CC (\calH_{f,G,i})^{p,q} \cdot t^{p-\frac{n}{2}}{\bar{t}}^{q-\frac{n}{2}}.
\end{equation}
\end{definition}

For the E-function considered in \cite{EGT} one has
\[
E(f,G)(t,\bar{t})= E^0(f,G)(t,\bar{t}) - E^1(f,G)(t,\bar{t}).
\]
One has the following relations between the ${\rm E}^i$-functions of a pair $(f,G)$ and the dual pair $(\widetilde{f}, \widetilde{G})$, which are refined versions of \cite[Theorem~9]{EGT}.

\begin{theorem} For $n$ even one has
\[ E^i(f,G)(t,\bar{t}) = E^i(\widetilde{f},\widetilde{G})(t^{-1},\bar{t}) \quad \mbox{for } i=0,1.
\]
For $n$ odd one has
\begin{eqnarray*}
E^0(f,G)(t,\bar{t})  & = &  E^1(\widetilde{f},\widetilde{G})(t^{-1},\bar{t}), \\
E^1(f,G)(t,\bar{t})  & = & E^0(\widetilde{f},\widetilde{G})(t^{-1},\bar{t}).
\end{eqnarray*}
\end{theorem}

\begin{proof} Let us define 
\[
E'(f,G)(t,\bar{t}) := E^0(f,G)(t,\bar{t}) + E^1(f,G)(t,\bar{t}).
\]
The arguments used in \cite{EGT} imply that the function $E'(f,G)(t,\bar{t})$ is given by the equations (2.8) and (2.9) without the sign $(-1)^{n_g}$ in the latter one. The computations presented in \cite[Section~4]{EGT} give the following equation for $E'(f,G)(t,\bar{t})$ (cf.\ \cite[Proposition~14]{EGT})
\begin{equation}
E(f,G)(t, \bar{t}) = \sum_{(g, \widetilde{g}) \in G \times \widetilde{G}}   \widehat{m}_{g,\widetilde{g}} (t \bar{t})^{{\rm age}(g)-\frac{n-n_g}{2}} \left( \frac{\bar{t}}{t} \right)^{{\rm age}(\widetilde{g})-\frac{n-n_{\widetilde{g}}}{2}},
\end{equation}
where the numbers $\widehat{m}_{g,\widetilde{g}}$ are defined in \cite{EGT}. Since $\widehat{m}_{\widetilde{g},g}=\widehat{m}_{g,\widetilde{g}}$, this equation implies the statement.
\end{proof}

\begin{corollary}
One has 
\[
\dim \calH_{f,G} = \dim \calH_{\widetilde{f},\widetilde{G}}
\]
$($and therefore ${\rm rk} \, \Lambda_{f,G} = {\rm rk} \, \Lambda_{\widetilde{f},\widetilde{G}}$$)$.
\end{corollary}

\begin{example} \label{Ex0}
We consider the invertible polynomial $f(x,y)=x^2y+y^5$ with its maximal group of symmetries $G=G_f$. Here $G_f$ is the group generated by the exponential grading operator $(\exp (2 \pi i)2/5), \exp((2 \pi i)1/5))$ and $(-1,0)$. The orbifold Milnor lattice is the direct sum
\[ \Lambda_{f,G} = \left( \bigoplus_{K \in {\rm Iso}\, G \atop K \neq \{ {\rm id} \}} \bigoplus_{g\in {\stackrel{\circ}{K}}} \calH_K^\ZZ \right) \oplus H_1(V_f; \ZZ)^G.
\]
We first show that $H_1(V_f; \ZZ)^G=0$. For this we consider a suitable real morsification of the function $f$ and the distinguished basis of vanishing cycles obtained by the method of N.~A'Campo and the second author from it (see, e.g., \cite[Section~4.1]{AGV2}). The corresponding Coxeter-Dynkin diagram is shown in Fig.~\ref{Fig0}. 
\begin{figure}
$$
\xymatrix{ *{\bullet}  \ar@{-}[dr]  \ar@{}^{2}[r] & & & &\\
& *{\oplus}  \ar@{-}[r]   \ar@{}^{5}[d] & *{\bullet} \ar@{-}[r]   \ar@{}^{3}[d] & *{\oplus} \ar@{-}[r] \ar@{}^{6}[d] & *{\bullet} \ar@{}^{4}[d]  \\
  *{\bullet} \ar@{-}[ur] \ar@{}_{1}[r]  & & & & 
  } 
$$
\caption{Coxeter-Dynkin diagram of $f(x,y)=x^2y+y^5$} \label{Fig0}
\end{figure}
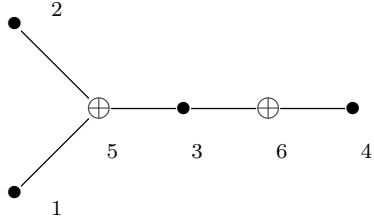
One can easily see that the elements of $H_1(V_f;\ZZ)$ which are invariant under the monodromy operator $J$ are linear combinations of the vanishing cycles corresponding to the saddle points (indicated by $\bullet$) such that the sum of the coefficients along the boundary of a region (indicated by $\oplus$ or $\ominus$) to which they are connected is equal to zero. Here these elements are generated by the basic elements
$1-2$ and  $1-3+4$.
None of them is invariant under the transformation $(x,y) \mapsto (-x,y)$ which corresponds to the reflexion at the horizontal axis. 
By our definition, the remaining part of the orbifold Milnor lattice $\Lambda_{f,G_f}$ with the orbifold intersection form is isomorphic to $A_1 \oplus A_4 \oplus A_4$. On the other hand, the dual polynomial $\widetilde{f}(x,y)=x^2+xy^5$ defines an $A_9$-singularity with the dual group $\widetilde{G}_f= \{ {\rm id}\}$.
\end{example}

We now examine two examples of Krawitz \cite[3.2]{Krawitz}.

\begin{example} \label{ExA}
We consider the invertible polynomial $f(x,y)=x^3y+xy^5$ of loop type (see \cite{KS}) with the group $G$ generated by the exponential grading operator $J=(\exp (2 \pi i)2/7), \exp((2 \pi i)1/7))$. The polynomial is self-dual and $\widetilde{G}=\langle (-1,-1) \rangle$. The orbifold Milnor lattice with respect to $G$ has a natural splitting
\[ \Lambda_{f,G} = {{\stackrel{\circ}{\Lambda}}} \oplus H_1(V_f; \ZZ)^G \mbox{ where } {{\stackrel{\circ}{\Lambda}}} =\bigoplus_{g \in G\setminus \{ {\rm id} \} } \calH_g^\ZZ.
\]
The lattice ${{\stackrel{\circ}{\Lambda}}}$ is isomorphic to $A_6$. In order to compute the invariant part $H_1(V_f; \ZZ)^G$ of the usual Milnor lattice of $f$, we proceed as in Example~\ref{Ex0}. A Coxeter-Dynkin diagram with respect to a suitable real morsification of the function $f$ is given by Fig.~\ref{FigA}.
\begin{figure}
$$
\xymatrix{ 
  *{\bullet} \ar@{-}[r] \ar@{}_{5}[d]  &  *{\ominus} \ar@{-}[r] \ar@{-}[d]  \ar@{--}[dr] \ar@{}_{1}[d] & *{\bullet} \ar@{-}[d] \ar@{-}[r] \ar@{}^{7}[d] & *{\ominus} \ar@{-}[r]  \ar@{-}[d] \ar@{--}[dl] \ar@{--}[dr] \ar@{}^{2}[d]  &*{\bullet} \ar@{-}[d] \ar@{}^{10}[d]  & & \\
 & *{\bullet} \ar@{-}[r] \ar@{}_{6}[d] &   *{\oplus} \ar@{-}[r] \ar@{-}[r] \ar@{-}[d] \ar@{--}[dr] \ar@{}_{14}[d]  & *{\bullet} \ar@{-}[r] \ar@{-}[d] \ar@{}_{9}[d]  & *{\oplus} \ar@{-}[r] \ar@{-}[d] \ar@{--}[dl]\ar@{--}[dr] \ar@{}_{15}[d] & *{\bullet} \ar@{-}[d] \ar@{}_{12}[d] &  \\
  &  & *{\bullet} \ar@{-}[r] \ar@{}_{8}[d] & *{\ominus} \ar@{-}[r] \ar@{}_{3}[d] & *{\bullet} \ar@{-}[r] \ar@{}_{11}[d]  & *{\ominus} \ar@{-}[r] \ar@{}_{4}[d] & *{\bullet} \ar@{}_{13}[d] \\
  & & & & & & } 
$$
\caption{Coxeter-Dynkin diagram of $f(x,y)=x^3y+xy^5$} \label{FigA}
\end{figure}
A basis of the subspace of invariant cycles is given by the elements
\[ 5-6+8-11+12, 5-7+9-11+13, 6-7+10-12+13 \]
The matrix for the Seifert form $L^G$ with respect to this basis is given by
\[ \left( \begin{array}{ccc} -5 & -2 & 2 \\
                                -2 & -5 & -2\\
                                2 & -2 & -5
\end{array} \right)
\]
It has determinant $-49$.

Now we consider the dual group $\widetilde{G}$. The orbifold Milnor lattice is
\[ \Lambda_{\widetilde{f},\widetilde{G}}= A_1 \oplus H_1(V_{\widetilde{f}}; \ZZ)^{\widetilde{G}}.
\]
The group $\widetilde{G}$ acts on the diagram of Fig.~\ref{FigA} by reflection at the central vertex 9. A basis of the subspace of $\widetilde{G}$-invariant cycles is given by
\[ 1+4, 2+3, 5+13, 6+12, 7+11, 8+10, 9, 14+15.
\]
The matrix of the Seifert form $L^{\widetilde{G}}$ with respect to this basis is given by
\[ \left( \begin{array}{cccccccc} -2 & 0 & -2 & -2 & -2 & 0 & 0 & 2\\
                                                   0 & -2 & 0 & 0 & -2 & -2 & -2 & 2\\
                                                   0 & 0 & -2 & 0 & 0 & 0 & 0 & 0\\
                                                   0 & 0 & 0 & -2 & 0 & 0 & 0 & -2\\
                                                   0 & 0 & 0 & 0 & -2 & 0 & 0 & -2\\
                                                   0 & 0 & 0 & 0 & 0 & -2 & 0 & -2\\
                                                   0 & 0 & 0 & 0 & 0 & 0 & -1 & -2\\
                                                   0 & 0 & 0 & 0 & 0 & 0 & 0 & -2
\end{array} \right)
\]
It has determinant 128.
\end{example}

\begin{example} \label{ExB}
We consider the invertible polynomial $f(x,y)=x^3y+y^4$ of chain type \cite{KS}, again with the group G generated by the exponential grading operator which is in this case $J=(\exp (2 \pi i)/4), \exp((2 \pi i)/4))$. The orbifold Milnor lattice of the pair $(f,G)$ again has a natural splitting
\[ \Lambda_{f,G} = {{\stackrel{\circ}{\Lambda}}} \oplus H_1(V_f; \ZZ)^G \mbox{ where } {{\stackrel{\circ}{\Lambda}}} =\bigoplus_{g \in G\setminus \{ {\rm id} \} } \calH_g^\ZZ.
\]
The lattice ${{\stackrel{\circ}{\Lambda}}}$ is in this case isomorphic to $A_1 \oplus A_1 \oplus A_1$. In order to compute the invariant part $H_1(V_f; \ZZ)^G$ of the usual Milnor lattice of $f$, we again proceed as in Example~\ref{Ex0}. A Coxeter-Dynkin diagram with respect to a suitable real morsification of the function $f$ is given in Fig.~\ref{FigB}.
\begin{figure}
$$
\xymatrix{ & & *{\bullet}  \ar@{-}[d]  \ar@{}_{2}[d] & & \\
& *{\bullet}  \ar@{-}[r]  \ar@{-}[d] \ar@{}_{3}[d] & *{\ominus} \ar@{-}[r]  \ar@{-}[d] \ar@{--}[dr]  \ar@{--}[dl] \ar@{}_{1}[d] & *{\bullet} \ar@{-}[d] \ar@{}_{4}[d] & \\
  *{\bullet} \ar@{-}[r] \ar@{}_{6}[d]  &  *{\oplus} \ar@{-}[r]  \ar@{}_{8}[d] & *{\bullet}  \ar@{-}[r] \ar@{}_{5}[d] & *{\oplus} \ar@{-}[r]  \ar@{}^{9}[d]  &*{\bullet} \ar@{}^{7}[d] \\
  & & & & } 
$$
\caption{Coxeter-Dynkin diagram of $f(x,y)=x^3y+y^4$} \label{FigB}
\end{figure}
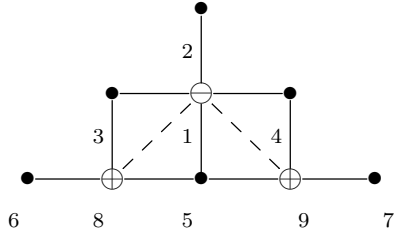
A basis of the subspace of invariant cycles is given by the elements
\[ 2-3+6, 2-4+7, 2-5+6 \]
The matrix of the Seifert form $L^G$ with respect to this basis is given by
\[ \left( \begin{array}{ccc} -3 & -1 & 1 \\
                                -1 & -3 & -1\\
                                1 & -1 & -3
\end{array} \right)
\]
It has determinant $-16$.

Now we consider the dual pair $(\widetilde{f},\widetilde{G})$.
The dual polynomial is $\widetilde{f}(x,y)=x^3+xy^4$ and the dual group is $\widetilde{G}=\langle (\exp (2 \pi i)1/3), \exp((2 \pi i)2/3)) \rangle$. The orbifold Milnor lattice is
\[ \Lambda_{\widetilde{f},\widetilde{G}}= A_2 \oplus H_1(V_{\widetilde{f}}; \ZZ)^{\widetilde{G}}.
\]
In order to compute the Seifert form $L^{\widetilde{G}}$ on $H_1(V_{\widetilde{f}}; \ZZ)^{\widetilde{G}}$, we work with the polynomial $h(x,y)=x^3+y^6$ which is in the same $\mu$-constant equivariant stratum as $\widetilde{f}(x,y)=x^3+xy^4$. A distinguished basis of vanishing cycles for this polynomial can be computed by the method of A.~M.~Gabrielov \cite{Gab}. It is obtained as follows: Let $e_1,e_2$ be a distinguished basis of vanishing cycles for $A_2$ ($x^3$) and $f_1, \ldots , f_5$ a distinguished basis of vanishing cycles for 
$A_5$ ($y^6$). Then 
\begin{equation} \label{eq:Gab}
\gamma_{ij} = e_i \otimes f_j
\end{equation}
is a distinguished basis of vanishing cycles for $h$. We extend these sets by $e_3:=-(e_1+e_2)$ and $f_6:=-(f_1 + \cdots + f_5)$. We extend the definition (\ref{eq:Gab}) to $i=3$ and $j=6$ as well. One has
\[ L(\gamma_{ij}, \gamma_{ij})=-1, L(\gamma_{ij}, \gamma_{i+1,j})=L(\gamma_{ij}, \gamma_{i,j+1})=1, L(\gamma_{ij}, \gamma_{i+1,j+1})=-1
\]
and $L(\gamma_{ij}, \gamma_{i'j'})=0$ otherwise, where $i+1=1$ for $i=3$ and $j+1=1$ for $j=6$. Then one can compute that the following cycles form a basis of the subspace of $\widetilde{G}$-invariant cycles:
\begin{eqnarray*}
b_{22} & = & \gamma_{22} + \gamma_{36}+\gamma_{14}, \\
b_{23} & = & \gamma_{23}+\gamma_{31}+\gamma_{15}, \\
b_{24} & = & \gamma_{24}+\gamma_{32}+\gamma_{16},\\
\delta & = & \gamma_{12}+\gamma_{13}+\gamma_{14}+\gamma_{15}+\gamma_{16}+\gamma_{22}+\gamma_{23}+\gamma_{24}+\gamma_{32}.
\end{eqnarray*}
The matrix of the Seifert form $L^{\widetilde{G}}$ with respect to this basis is given by
\[ \left( \begin{array}{cccc} -3 & 3 & 3 & 3 \\
0 & -3 & 3 & 0\\
0 & 0 & -3 & -3 \\
0 & 0 & 0 & -1
\end{array} \right) .
\]
It has determinant $27$.
\end{example}


\bigskip
\noindent Leibniz Universit\"{a}t Hannover, Institut f\"{u}r Algebraische Geometrie,\\
Postfach 6009, D-30060 Hannover, Germany \\
E-mail: ebeling@math.uni-hannover.de\\

\medskip
\noindent Moscow State University, Faculty of Mechanics and Mathematics,\\
Moscow, GSP-1, 119991, Russia\\
E-mail: sabir@mccme.ru

\begin{thebibliography}{10}

\bibitem{AGV2} V.~I.~Arnold,  S.~M.~Gusein-Zade,  A.~N.Varchenko:
Singularities of Differentiable Maps, Volume II. Birkh\"auser, Boston--Basel--Berlin, 1988.

\bibitem{BH1} P.~Berglund, T.~H\"ubsch: A generalized construction of
mirror manifolds. Nuclear Physics B {\bf 393} (1993), 377--391. 

\bibitem{BH2} P.~Berglund, M.~Henningson: Landau-Ginzburg orbifolds,
mirror symmetry and the elliptic genus. Nuclear Physics B {\bf 433}
(1995), 311--332.

\bibitem{EG-Edinburgh} W.~Ebeling, S.~M.~Gusein-Zade:
Orbifold zeta functions for dual invertible polynomials.
Proc. Edinb. Math. Soc. {\bf 60} (2017), no.1, 99--106.

\bibitem{EGT} W.~Ebeling, S.~M.~Gusein-Zade, A.~Takahashi:
Orbifold E-functions of dual invertible polynomials. 
Journ. Geom. Phys. {\bf 106} (2016), 184--191.

\bibitem{ET} W.~Ebeling, A.~Takahashi: Mirror symmetry between
orbifold curves and cusp singularities with group action. Int. Math. Res. Not. {\bf 2013} (2013),
2240--2270.

\bibitem{Ruan_etal} H.~Fan, T.~Jarvis, Y.~Ruan: The Witten equation,
mirror symmetry, and quantum singularity theory. Ann. of Math. (2) {\bf 178} (2013), no.1, 1--106.

\bibitem{Gab} A.~M.~Gabrielov:
Intersection matrices for certain singularities. 
Funkcional. Anal. i Prilozen. {\bf 7} (1973), no. 3, 18--32. (Engl. translation in Functional Anal. Appl. {\bf 7} (1974), 182--193.)

\bibitem{Ito-Reid}  Y.~Ito, M.~Reid:
The McKay correspondence for finite subgroups of ${\rm SL}(3,\CC)$.
In: Higher-dimensional complex varieties (Trento, 1994), 221--240, de Gruyter, Berlin, 1996. 

\bibitem{Krawitz} M.~Krawitz: FJRW-rings and Landau--Ginzburg mirror symmetry. 
Preprint arXiv: 0906.0796.

\bibitem{KS} M.~Kreuzer, H.~Skarke: On the classification of quasihomogeneous functions.
Commun. Math. Phys. {\bf 150} (1992), 137--147.
 
\bibitem{Zaslow}  E.~Zaslow: Topological orbifold models and quantum cohomology rings.
Comm. Math. Phys. {\bf 156} (1993), no.2, 301--331.

\end{thebibliography}
\end{document}